\documentclass[12pt]{amsart}

\usepackage{amssymb}

\usepackage{enumerate}

\usepackage{graphicx}

\usepackage[all]{xy}
\usepackage{url}
\usepackage{multirow}

\makeatletter
\@namedef{subjclassname@2010}{%
  \textup{2010} Mathematics Subject Classification}
\makeatother

\newtheorem{thm}{Theorem}[section]
\newtheorem{cor}[thm]{Corollary}
\newtheorem{lem}[thm]{Lemma}
\newtheorem{prop}[thm]{Proposition}
\newtheorem{conj}[thm]{Conjecture}

\theoremstyle{definition}

\newtheorem{Def}[thm]{Definition}

\numberwithin{equation}{section}

\newcommand{\Z}{\mathbb{Z}}
\newcommand{\R}{\mathbb{R}}
\newcommand{\Q}{\mathbb{Q}}
\newcommand{\C}{\mathbb{C}}

\newcommand{\e}{\varepsilon}
\newcommand{\p}{\mathfrak{p}}
\renewcommand{\O}{\mathcal{O}}

\DeclareMathOperator{\id}{id}
\DeclareMathOperator{\rk}{rank}
\DeclareMathOperator{\Nm}{Nm}
\DeclareMathOperator{\Emb}{Emb}
\DeclareMathOperator{\Gal}{Gal}

\begin{document}


\baselineskip=17pt



\title[A Variation on Leopoldt's Conjecture]{A Variation on Leopoldt's Conjecture: Some Local Units instead of All Local Units}

\author[D.~Nelson]{Dawn Nelson}
\address{Department of Mathematics\\ Bates College\\
Lewiston Maine, USA}
\email{dnelson@bates.edu}

\date{}

\begin{abstract}
{Leopoldt's Conjecture is a statement about the relationship between the global and local units of a number field. Approximately the conjecture states that the $\Z_p$-rank of the diagonal embedding of the global units into the  product of {\em all} local units equals the $\Z$-rank of the global units. The variation  we consider asks: Can we say anything about the $\Z_p$-rank of the diagonal embedding of the global units into the product of {\em some} local units? We use the $p$-adic Schanuel Conjecture to answer the question in the affirmative and moreover we give a value for the $\Z_p$-rank (of the diagonal embedding of the global units into the product of {\em some} local units) in terms of the $\Z$-rank of the global units and a property of the the local units included in the product.}
\end{abstract}

\subjclass[2010]{Primary 11R27; Secondary 11R99}

\keywords{Leopoldt's Conjecture, units}

\maketitle

\section{Introduction}

H.W.~Leopoldt proposed his conjecture relating global and local units in 1962 \cite{Leo}. 
 Since then his  conjecture has been much studied. It has been generalized, strengthened, weakened, and continues to be actively studied because of its connections to other areas of number theory. We continue the tradition of considering variations on Leopoldt's Conjecture.

Leopoldt's Conjecture appears in the study of $p$-adic zeta-functions \cite{Col} \cite{Sfp}, $K$-theory \cite{Kol}, and Iwasawa theory \cite{Gil}  \cite{NSW}. In particular, 
 Leopoldt's Conjecture is related to the splitting of exact sequences of  Iwasawa modules \cite{KW} \cite{Wint}. The conjecture is also equivalent to computing the dimension of a certain Galois cohomology group  \cite{NSW} \cite{NQD}. A variation on the conjecture has been used in Galois deformation theory \cite{CM}.

Throughout this paper we will consider Galois number fields $M$ where $\O_M$ is the ring of integers of $M$ and $\O_M^*$ is the group of {\em global units} in $\O_M$.  
We will also fix a prime number $p$ and use $\p$ to denote prime ideals in $\O_M$ above $p$. Then
 $M_\p$ is the completion of $M$ with respect to the non-archimedean $\p$-adic value,  $\O_\p$ is the ring of integers of $M_\p$, and $\O_\p^*$ is the group of {\em local units}. Let $\pi$ be the uniformizer of $\p$ in $\O_\p$  
and define the {\em principal local units}: $\O^*_{\p,1} := 1+\pi\O_\p$.

\pagebreak
We can state Leopoldt's Conjecture informally:
\begin{conj}[Leopoldt]
 Define the diagonal embedding:
\begin{align}\Delta:\O_M^*&\rightarrow \prod_{\p|p}\O^*_{\p}\notag\\
		u&\mapsto(u,\ldots,u).\notag
\end{align}
 Then $\rk_\Z \O_M^*=\rk_{\Z_p}(\Delta\O_M^*)$.
\end{conj}
\noindent See Section \ref{2ways} for the precise formulation which involves principal local units and a topological closure.

Leopoldt's Conjecture has been proven in some special cases. For example, using a method outlined by J.~Ax \cite{Ax} and A.~Brumer's \cite{Br} $p$-adic version of a theorem of A.~Baker \cite{Ba}, one can prove that the conjecture is true for abelian extensions of $\Q$, for CM fields with abelian maximal real subfields, and for abelian extensions of imaginary quadratic fields. A paper by M.~Laurent proves that the conjecture is true for Galois extensions that satisfy certain conditions on absolutely irreducible characters of the Galois group \cite{Lau}.
 In 2009, P.~Mih\u{a}ilescu announced a proof for all number fields \cite{Mih}.

The variation we consider concerns the map:
\[\Delta_\Gamma:\O_M^*\rightarrow\prod_{\p\in \Gamma}\O^*_{\p},\]
where $\Gamma$ is a subset of the primes above the fixed $p$, $\Gamma\subset\{\p\subset \O_M\,\big|\,\p|p\}$.
We explore the question: Can we say anything about $\rk_{\Z_p}(\Delta_\Gamma\O_M^*)$? The answer: Yes!

Our strongest result can be stated informally:
\begin{thm}
Assume both Schanuel's Conjecture and the $p$-adic version. Let $t$ be a constant determined by which $\p$ are in $\Gamma$.\footnote{See Section \ref{calc} for details on $t$.} Let $M$ be real or CM over $\Q$. Then $\rk_{\Z_p}(\Delta_\Gamma\O_M^*) = \rk_\Z \O_M^*- t +1.$
\end{thm}
\noindent In the case of complex, non-CM, extensions we have the weaker result:
\begin{thm}
Assume the $p$-adic Schanuel Conjecture. Let $t$ be a constant determined by which $\p$ are in $\Gamma$. Let $M$ be a complex, non-CM, extension of  $\Q$. Then $\rk_{\Z_p}(\Delta_\Gamma\O_M^*) \geq \rk_\Z \O_M^*- t +1.$
\end{thm}
\noindent See Section \ref{calc} for the formal statements of these theorems which involve principal local units and topological closures.

As we know from linear algebra the rank of an image can be expressed as the rank of a matrix. Thus in this paper we consider an appropriate matrix, one whose entries are $p$-adic logarithms, and we use transcendence theory (in particular Schanuel's Conjecture) to calculate its rank.

\section{Notation}\label{nota}

 For $q\in\Q_p$ let $|q|_p$ be the usual $p$-adic absolute value. It has a unique extension to the algebraic closure of $\Q_p$ and to $\C_p$, the completion of the algebraic closure  of $\Q_p$. By abuse of notation, we also let $|\cdot|_p$  be the absolute value on that extension.

 For $\{x\in \C_p:|x-1|_p<1\}$, define the $p$-adic logarithm $\log_p(x)$ by the usual power series 
\[\log_p(X)=\sum_{n=1}^{\infty}\frac{(-1)^{n+1}(X-1)^n}{n}.\]
We  can uniquely extend $\log_p$  to all of $\C_p^*$ so that $\log_p(ab)=\log_p(a)+\log_p(b)$ and $\log_p(\zeta p^s)=0$ for all roots of unity $\zeta$ and $s\in \Z$.   
 For $\{x\in \C_p:|x|_p<p^{-1/(p-1)}\}$, define the $p$-adic exponential $\exp_p(x)$ by the usual power series 
\[\exp_p(X)=\sum_{n=1}^{\infty}\frac{(X)^n}{n!}.\]
The $p$-adic exponential does not extend uniquely to all of $\C_p$.

 Note that $\log_p$ is injective on $\{x\in \C_p:|x-1|_p<1\}$. Moreover, the function $\log_p$ gives an isomorphism between the multiplicative group $\{x\in \C_p:|x-1|_p<p^{-1/(p-1)}\}$ and the additive group $\{x\in \C_p:|x|_p<p^{-1/(p-1)}\}$. The function $\exp_p$ is the inverse of $\log_p$ on these groups.

Let $G=\Gal(M/\Q)$ and $|G|=n$.  Define $E_p:=\Emb(M,\C_p)$ to be the set of all embeddings of $M$ into $\C_p$. Similarly define $E:=\Emb(M,\C)$ to be the set of all embeddings of $M$ into $\C$. For $\tau\in E$, we can define $\overline{\tau}\in E$ so that $\overline{\tau}(m) = \overline{\tau(m)}$ where $\overline{\tau(m)}$ is the complex conjugate of ${\tau}(m)$.

The relationships between $G$, $E$, and $E_p$ will be important in what follows. Although $G$ is a group and $E$ and $E_p$ are only sets, all three have the same cardinality. For any fixed $\tau\in E$ we have $E=\{\tau\circ g\,|\,g\in G\}$. Similarly, for any fixed $\sigma\in E_p$  we have $E_p=\{\sigma\circ g\,|\,g\in G\}$. Moreover, the following lemma  defines natural  bijections between $E$ and $E_p$.

\begin{lem}\label{bij} Let $\psi:\C_p\rightarrow\C$ be an isomorphism.\footnote{The proof that such an isomorphism exists requires Zorn's Lemma (equivalently, the Axiom of Choice). In fact, there are infinitely many isomorphisms between $\C_p$ and $\C$.} Then there is a bijection from $E_p$ to $E$ given by 
\[(\psi.\sigma)(m):=\psi(\sigma(m))\]
where $\sigma\in E_p$ and $m\in M$. Similarly, given an isomorphism from $\C$ to $\C_p$ there is a bijection from $E$ to $E_p$.
\end{lem}

Throughout this paper we will rely on the existence of a special global unit. Because of the similarity to Minkowski units as described in \cite{Nar}, we  call this unit a {\em weak Minkowski unit}. In fact, in the case of real extensions, our definition and the usual definition agree.

\begin{Def} An element $\e$ in $\O_M^*$ is a {\em weak Minkowski unit} if the $\Z$-module generated by $g\e$, for all $g \in \Gal(M/\Q)$, is of finite index in $\O_M^*$.
\end{Def}

{Under the usual definition, there are extensions for which a Minkowski unit does not exist and there are extensions for which it is not yet known whether or not a Minkowski unit exists \cite{Nar}. On the other hand the case for weak Minkowski units is settled.}

\begin{prop}\label{unit} For all finite Galois extensions $M$ of $\Q$ there exists a weak Minkowski unit. Moreover, we can find a weak Minkowski unit $\e$ such that   $|\sigma_i g_j \e-1|_p<1$ for all $g_j\in G$ and $\sigma_i\in E_p$.
\end{prop}

\begin{proof}
This proof uses many of the same steps as  one proof  of Dirichlet's Unit Theorem (see \cite[Theorem 5.9]{ANT}). 

Let $r=\rk_\Z\O_M^*$ and $|G|=n$. Dirichlet's Unit Theorem relates $r$ and $n$:
\[r=\begin{cases}
     n-1 & \text{if }M\text{ is real}, \\
     \frac{n}{2}-1 & \text{if }M\text{ is complex}.
\end{cases}
\]
If $M/\Q$ is totally real take $E=\Emb(M,\R)= \{\tau_1,\ldots,\tau_{r+1}\}$. If $M/\Q$ is complex  take  $E=\Emb(M,\C)= \{\tau_1,\ldots,\tau_{r+1},\overline{{\tau}_1},\ldots,\overline{{\tau}_{r+1}}\}$. Number the  $g\in G$ so that $\tau_1=\tau_ig_i$ and $\tau_1=\overline{{\tau}_i}{g}_{r+1+i}$. 
Consider the map $\hat\tau$:
\[\begin{array}{|rcl|rcl|}\hline
\multicolumn{3}{|l|}{\mbox{if }M/\Q\mbox{ is real}}&\multicolumn{3}{|l|}{\mbox{if }M/\Q\mbox{ is complex}}\\\hline
\hat\tau:M&\rightarrow &\R^{r+1}=:X&\hat\tau:M&\rightarrow &\C^{r+1}=:X\\
\alpha&\mapsto&(\tau_1\alpha,\ldots,\tau_{r+1}\alpha)&\alpha&\mapsto&(\tau_1\alpha,\ldots,\tau_{r+1}\alpha)\\\hline
\end{array}\]
 
 Let $\vec{x}=(x_1,\ldots, x_{r+1})$ be in $X$. Define  
\[|\Nm(\vec{x})|:=\begin{cases}
   \prod|x_i|   & \text{if }M\text{ is real}, \\
     \prod|x_i|^2 & \text{if }M\text{ is complex}.
\end{cases} 
\]
This definition agrees with the usual definition of the norm of an element in a number field, i.e., for $\alpha\in\O_M$, $|\Nm(\hat\tau(\alpha))|=|\Nm_{M/\Q}(\alpha)|$. In what follows, we always take $\vec{x}$ to be in the set $X':=\left\{\vec{x}\in X \,\Big|\, \frac{1}{2}\leq|\Nm(\vec{x})|\leq1\right\}$. For $\vec{x},\vec{y}\in X$, define $\vec{x}\cdot\vec{y}$ to be componentwise multiplication and $\vec{x}\cdot\hat\tau(\O_M):=\{\vec{x}\cdot\hat\tau(\alpha)\,|\,\alpha\in\O_M\}$. 
 
In $X$, $\hat\tau(\O_M)$ is a full lattice with  fundamental domain having finite calculable volume  $V$ \cite[Proposition 4.26]{ANT}. Moreover,  $\vec{x}\cdot\hat\tau(\O_M)$ is a full lattice in $X$ with fundamental domain having volume $V'(\vec{x})=V|\Nm(\vec{x})|$ and $V'(\vec{x})\leq V$. Take a subset, $T\subset X$, that is compact, convex, and symmetric with respect to the origin and such that vol$(T)\geq2^mV'(\vec{x})$ for all $\vec{x}\in X'$ ($m=\dim_\R X$). 

By Minkowski's Theorem,\footnote{{\bf Minkowski's Theorem.} \em Let $D$ be the fundamental domain of a lattice.  Let $T$ be a subset of a real vector space of dimension m that is compact, convex, and symmetric in the origin. If vol$(T)\geq2^m\mbox{vol}(D)$ then T contains a point of the lattice other than the origin.} for all $\vec{x}\in X'$ there exists a nonzero $\beta\in\O_M$  such that $\vec{x}\cdot\hat\tau(\beta)\in T$.  Points in $T$ have bounded coordinates and thus bounded norm,  so   for some fixed $N\in\R$
 \[|\Nm(\vec{x}\cdot\hat\tau(\beta))|\leq N, \]
and
\[|\Nm_{M/\Q}(\beta)|= |\Nm(\hat\tau(\beta))|\leq N/|\Nm(\vec{x})|\leq 2N.
 \]
Now consider all ideals $\beta\O_M$ where $\beta$ is such that there is some $\vec{x}\in X'$ with $\vec{x}\cdot\hat\tau(\beta)\in T$. The norm of $\beta$ is bounded, thus there are only finitely many such ideals, call them $\{\beta_1\O_M,\ldots,\beta_t\O_M\}$. So for any $\beta$ with $\vec{x}\cdot\hat\tau(\beta)\in T$ we have $\beta\O_M=\beta_j\O_M$ for some $j$. Hence there exists a unit $\epsilon$ such that  $\beta=\beta_j\epsilon$ and $\vec{x}\cdot\hat\tau(\epsilon)\in \hat\tau(\beta_j^{-1})\cdot T$. 

Define $T':=\hat\tau(\beta_1^{-1})T\cup\ldots\cup\hat\tau(\beta_t^{-1})T$, it  is bounded and independent of any $\vec{x}$. The previous paragraph shows that for each $\vec{x}\in X'$ there exists a unit $\epsilon$ such that $\vec{x}\cdot\hat\tau(\epsilon)\in T'$ and hence $\vec{x}\cdot\hat\tau(\epsilon)$ has bounded coordinates independent of $\vec{x}$.

We are now ready to construct our weak Minkowski unit. Choose $\vec{x}$ so that for $k\neq 1$ the coordinates $x_k$ are very large compared to those of $T'$  and $x_1$ is very small so that $|\Nm(\vec{x})|=1$. This $\vec{x}$ is in $X'$, so there exists a unit $\e$ (which will be our weak Minkowski unit) such that $\vec{x}\cdot\hat\tau(\e)\in T'$ and hence has bounded coordinates, i.e., $|x_k\tau_k(\e)|\leq L$ for some $L\in \R$. For $k\neq1$, we chose  $x_k$  large enough so that  $|\tau_k(\e)|<1$. Hence $|\tau_ig_j(\e)|<1$ if $\tau_ig_j\neq\tau_1$, i.e., if  $i\neq j$. So $\log|\tau_ig_j(\e)|<0$ for $i\neq j$.

 For any fixed  $j$, $g_j\e\in\O_M^*$ hence   $\displaystyle\sum_{i=1}^{r+1}\log|\tau_ig_j(\e)|=0$. So for $j\neq r+1$
\[\displaystyle\sum_{i=1}^{r}\log|\tau_ig_j(\e)|=-\log|\tau_{r+1}g_j(\e)|>0.\]
Combining this with the fact that $\log|\tau_ig_j(\e)|<0$ for $i\neq j$, linear algebra tells us that   the $r\times r$ matrix $(\log|\tau_ig_j(\e)|)_{i,j=1,\ldots, r}$ is invertible. 

	The invertibility of the matrix implies that $\{g_1\e,\ldots,g_{r}\e\}$ are multiplicatively independent. 
Since the $\Z$-rank of $\O_M^*$ is $r$, the $r$ multiplicatively independent elements $\{g_1\e,\ldots,g_{r}\e\}$ generate a $\Z$-module of finite index in $\O_M^*$. Moreover, the $\Z$-module generated by $g\e$, for all $g\in G$, also has finite index in $\O_M^*$. Thus $\e$ is a weak Minkowski unit.

By Lemma \ref{FLT} below, there exists an $s$ such that such that   $|\sigma_i g_j \e^s-1|_p<1$ for all $g_j\in G$ and $\sigma_i\in E_p$. If $\e$ is a weak Minkowski unit then so is $\e^{s}$. Thus, if the first weak Minkowski unit we find does not satisfy the desired condition, we can replace it with one that does.
\end{proof}

\begin{lem}\label{FLT}
If $u\in\O^*_{M}$ then for all $\sigma\in E_p$, $\left|\sigma \left(u^{p^f-1}\right)-1\right|_p<1$, where $f=[\O_M/\p\O_M:\Z/p\Z]$ and $\p$ is a prime ideal above $p$.
\end{lem}

Lemma \ref{FLT} is a relative of Fermat's Little Theorem.

\begin{cor}\label{r}
Let $M$ be a finite Galois extension of $\Q$. Let $\e$ be a weak Minkowski unit. Then
$\rk\left(\log|\tau_ig_j(\e)|\right)_{\tau_i\in E,\,g_j\in G}=\rk_\Z O_M^*.$
\end{cor}

\begin{proof}
In Proposition \ref{unit} we showed that for specially numbered $\tau_i$ and $g_j$ the matrix $(\log|\tau_ig_j\e|)_{i,j=1,\ldots,r}$ is invertible. Thus it has rank equalling $r:=\rk_\Z O_M^*$. If we increase the size of the matrix by one row and one column to $(\log|\tau_ig_j\e|)_{i,j=1,\ldots,r+1}$, the rank remains $r$ because for fixed $j$ \[\displaystyle\sum_{i=1}^{r+1}\log|\tau_ig_j(\e)|=0.\] So in the case of totally real extensions, we have finished the proof. In the case of a complex extension, when we change the matrix to $\left(\log|\tau_ig_j(\e)|\right)_{\tau_i\in E,\,g_j\in G}$ each new row is dependent on a previous row because, for any $\tau\in E$, $x\in M$, $|\tau(x)||\bar{\tau}(x)|=|\tau(x)|^2$. Thus the rank remains $r$. Finally we note that swapping rows or columns does not affect the rank of a matrix, hence we can number the rows and columns of $\left(\log|\tau_ig_j(\e)|\right)_{\tau_i\in E,\,g_j\in G}$ independently of the special numbering used in the proof of the proposition.
\end{proof}

{  In what follows  $c$ will always denote an element of $G$ that is induced by complex conjugation}, i.e., the elements defined in Lemma \ref{com} below.

  \begin{lem}\label{com}
  Let $M$ be a complex Galois extension of $\Q$. Then for all $\tau\in E$ there exists $c\in G$ such that the following diagram commutes.  \[\xymatrix{M\ar_{c}[d]\ar^{\overline{\tau}}[dr]\\
M\ar_{\tau}[r]&\C
}\]
Moreover elements induced by complex conjugation have the following properties. 
\begin{itemize}
\item Let $c$ and $c'$ be elements of $G$ induced by complex conjugation, then there exists an $h\in G$ such that $c=hc'h^{-1}$. 
\item Also, for any $c$ and $h$ in $G$,  $hch^{-1}$ is induced by complex conjugation.
  \end{itemize}  
    \end{lem}

\section{Leopoldt's Conjecture: Two Ways}\label{2ways}

We state two precise versions of Leopoldt's Conjecture that are equivalent for Galois number fields.

Let $\Delta$ be the diagonal embedding of the {global} units $\O_M^*$ into the product of local units $\prod_{\p| p}\O_\p^*$ and define $X:=\Delta^{-1}\left(\prod_{\p|p}\O_{\p,1}^*\right)$, i.e., $X$ is the inverse image of the product of the principal local units. Let  $\overline{\Delta(X)}$  denote the topological closure of $\Delta(X)$ in $\prod_{\p|p}\O_{\p,1}^*$.  

\begin{conj}[Leopoldt]\label{DX}
Let $M$ be a number field. Then  the $\Z$-rank of $\O_M^*$ equals the $\Z_p$-rank of $\overline{\Delta(X)}$.
\end{conj}

Leopoldt's Conjecture can also be stated in a way that contains a matrix with $p$-adic logarithms and in so doing partially reveals the conjecture's relationship with the $p$-adic regulator.

\begin{conj}[Leopoldt]\label{mat}
Let  $M$ be a finite Galois extension of $\Q$, with  $g_j\in G=\Gal(M/\Q)$ and $\sigma_i\in E_p:=\Emb(M,\C_p)$. 
Let
 $\e\in \O_M^*$ be such that $\{g_1\e,\ldots,g_n\e\}$ generates a finite index subgroup of $\O_M^*$ and such that $|\sigma_i g_j \e-1|_p<1$ for all $i,j\in\{1,\ldots,n\}$.
 Then \[\rk\left(\log_p(\sigma_i g_j \e)\right)_{\sigma_i\in E_p,\,g_j\in G}=\rk_\Z \O_M^*.\]
\end{conj}

Since $\rk_\Z \O_M^*=\rk\left(\log|\tau_i g_j\e|\right)_{{\tau_i\in E},\,g_j\in G}$, the conclusion of Leopoldt's Conjecture can be stated in terms of the equality of the ranks of two matrices 
\[\rk\left(\log_{{p}}(\sigma_i g_j\e)\right)_{{\sigma_i\in E_p},\,g_j\in G}
=\rk\left(\log|\tau_i g_j\e|\right)_{{\tau_i\in E},\,g_j\in G}.\]
Note that the entries of the matrix on the left are $p$-adic logarithms, while the entries on the right are obtained by first taking the complex modulus and then taking the real logarithm. Also notice that the rows of the left matrix are indexed by elements in $E_p$ whereas the rows of the right matrix are indexed by elements in $E$.

\section{The Variation}\label{var}

Our variation can be stated precisely in a manner similar to Conjecture \ref{DX}.
Let $\Gamma\subset\{\p\subset \O_M\,\big|\,\p|p\}$ and $\Delta_\Gamma$ be the diagonal embedding
\[\Delta_\Gamma:\O_M^*\rightarrow\prod_{\p\in \Gamma}\O^*_{\p}.\]
Define $Y:=\Delta_\Gamma^{-1}\left(\prod_{\p\in\Gamma}\O_{\p,1}^*\right)$, i.e., $Y$ is the inverse image of the product of some principal local units.
Let  $\overline{\Delta_\Gamma(Y)}$  denote the topological closure of $\Delta_\Gamma(Y)$ in $\prod_{\p\in\Gamma}\O_{\p,1}^*$.  
We will consider the value of $\rk_{\Z_p}(\overline{\Delta_\Gamma(Y)})$. Equivalently, mimicking the relationship between Conjectures \ref{DX} and \ref{mat}, we will consider  the rank of a matrix whose entries are $p$-adic logarithms. In fact, the matrix we consider will be the one from Conjecture \ref{mat} with rows removed, i.e., with the $\sigma_i$ in a subset of $E_p$.

Fix $\p_0\in\Gamma$. Call the map $M\rightarrow M_{\p_0}\rightarrow \C_p$, $\sigma_0$. It is in $E_p$. Next define the surjection
\begin{align}
E_p&\rightarrow\{\p|p\}\notag\\
\sigma=\sigma_0g&\mapsto g^{-1}\p_0\notag
\end{align}
Finally, define $S_p:=\{\sigma=\sigma_0g\,|\,g^{-1}\p_0\in\Gamma\}$. 

\begin{prop}\label{equivV}
Let $M$ be a finite Galois extension of $Q$. Let $\e$ be a weak Minkowski unit such that   $|\sigma_i g_j \e-1|_p<1$ for all $g_j\in G$ and $\sigma_i\in E_p$. Keep the definitions and notations defined above. Then
\[\rk_{\Z_p}(\overline{\Delta_\Gamma(Y)})=\rk\left(\log_p(\sigma_i g_j\e)\right)_{\sigma_i\in S_p,\,g_j\in G}.\]
\end{prop}

Note that when $\Gamma$ contains all primes above $p$ then $S_p=E_p$, and this proposition claims that the two versions of Leopoldt's Conjecture stated above are equivalent.

We will need a lemma.
\begin{lem}\label{conv} Let $M$ be a finite Galois extension of $\Q$. Let $u_1,\ldots, u_t\in \O_M^*$ and $a_1,\ldots,a_t\in \Z_p$ with at least one non-zero. The product $u_1^{a_1}\cdots u_t^{a_t}=1$ in $M_\p$ for all $\p\in\Gamma$ ($\Gamma$ is any non-empty subset of primes above $p$) if and only if, for a fixed place $\p_0\in\Gamma$, $(hu_1)^{a_{1}}\cdots (hu_t)^{a_{t}}=1$ in $M_{\p_0}$ for all $h\in H:=\{h\in G\,|\,h\p=\p_0 \mbox{ for some }\p\in\Gamma\}$.
\end{lem}

\begin{proof}
If $u_1^{a_1}\cdots u_t^{a_t}=1$ in $M_\p$ for all $\p\in\Gamma$ then there exists $a_{j,n}\in \Z$ such that $p$-adically $a_{j,n}\rightarrow a_j$ and $u_1^{a_{1,n}}\cdots u_t^{a_{t,n}}\rightarrow1$ in $M_\p$ for all $\p\in\Gamma$.
Recall that the absolute values corresponding to each prime are related by $|x|_{\p_i}=|gx|_{\p_j}$ for $g\in G$ such that $\p_j=g\p_i$. Thus  $(hu_1)^{a_{1,n}}\cdots (hu_t)^{a_{t,n}}\rightarrow1$ in $M_{\p_0}$ for all $h\in H$. Hence $(hu_1)^{a_{1}}\cdots (hu_t)^{a_{t}}=1$ in $M_{\p_0}$ for all $h\in H$.

The argument for the converse is similar and is left to the reader.
\end{proof}

\begin{proof}[Proof of Proposition \ref{equivV}]

First we prove that \[\rk_{\Z_p}(\overline{\Delta_\Gamma(Y)})\geq\rk\left(\log_p(\sigma_i g_j\e)\right)_{\sigma_i\in S_p,\,g_j\in G}.\] 

Let \[\rk\left(\log_p(\sigma_i g_j\e)\right)_{\sigma_i\in S_p,\,g_j\in G}=t.\] 
For contradiction assume \[\rk_{\Z_p}(\overline{\Delta_\Gamma(Y)})<t.\] Then for all $t$-sized subsets  $\{g_1,\ldots,g_t\}\subset G$, the $\{\Delta(g_1\e),\ldots, \Delta(g_t\e)\}$ are $\Z_p$-dependent. So there exists at least one non-zero $a_j\in\Z_p$ such that $(g_1\e)^{a_1}\cdots(g_t\e)^{a_t}=1$ in $M_\p$ for all $\p\in\Gamma$. Thus for a fixed $\p_0$, Lemma \ref{conv} guarantees that   
\begin{equation}\label{Prod}
(hg_1\e)^{a_{1}}\cdots (hg_t\e)^{a_{t}}=1
\end{equation}
 in $M_{\p_0}$ for all $h\in H$. 

As above, $\sigma_0$ is the map $M\rightarrow M_{\p_0}\rightarrow \C_p$.
Apply the $p$-adic logarithm to Equation \ref{Prod}:
\begin{equation}\label{22}
\sum_{j=1}^t a_j\log_p(\sigma_0 hg_j\e)=0
\end{equation}
 for all $h\in H$.
Since the defining conditions on $H$ and $S_p$ are related, Equation \ref{22} becomes
\[\sum_{j=1}^t a_j\log_p(\sigma_i g_j\e)=0\]
 for all $\sigma_i\in S_p.$
Recall that this was true for all $t$-sized subsets of $G$. So we have contradicted the  fact that $\rk\left(\log_p(\sigma_i g_j\e)\right)_{\sigma_i\in S_p,\,g_j\in G}=t$.

Second we prove that \[\rk_{\Z_p}(\overline{\Delta_\Gamma(Y)})\leq\rk\left(\log_p(\sigma_i g_j\e)\right)_{\sigma_i\in S_p,\,g_j\in G}.\]

Let $\rk_{\Z_p}(\overline{\Delta_\Gamma(Y)})=t$. So there exists some $t$-sized subset of $\{g_1\e,\ldots,g_n\e\}$ for which the elements are $\Z_p$-independent in $\overline{\Delta_\Gamma(Y)}$. For illustration purposes we will assume that  $g_1\e,\ldots,g_t\e$ are $\Z_p$-independent.

 For contradiction  suppose $\rk\left(\log_p(\sigma_i g_j\e)\right)_{\sigma_i\in S_p,\,g_j\in G}<t$. 
Then there exists at least one non-zero $a_j\in \C_p$ such that 
\[\sum_{j=1}^t a_j\log_p(\sigma_i g_j\e)=0 \mbox{ for all }\sigma_i\in S_p.\]
Hence with $\p_0$ and $\sigma_0$ fixed as above we have
\begin{equation}\label{eq1}\sum_{j=1}^ta_j\log_p(\sigma_0h g_j\e)=0 \mbox{ for all }h\in H.\end{equation}
Recall 
\begin{align*}
H&:=\{h\in G\,|\,h\p=\p_0 \mbox{ for some }\p\in\Gamma\}\\
&\phantom{:}=\{h\in G\,|\,h^{-1}\p_0=\p \mbox{ for some }\p\in\Gamma\}.  
\end{align*}

Let $D_{\p_0}$ be the decomposition group of $\p_0$. Its elements permute those of $H$. Let $d\in D_{\p_0}$. Then $d h\in H$ because $h^{-1}d^{-1}\p_0=h^{-1}\p_0\in\Gamma$.
 Apply $d$ to Equation (\ref{eq1}), note that Galois elements commute with the $p$-adic logarithm,
\begin{equation*}\sum_{j=1}^t d a_j\log_p(d\sigma_0h g_j\e)=0\mbox{ for all }h\in H,\end{equation*}
and
\begin{equation*}\label{eq2}\sum_{j=1}^t d a_j\log_p(\sigma_0 hg_j\e)=0\mbox{ for all }h\in H.\end{equation*}
 Letting $d$ vary we have
\begin{equation}\label{eq3}
\sum_{d\in D_{\p_0}} \sum_{j=1}^t d a_j\log_p(\sigma_0 hg_j\e)=\sum_{j=1}^t  \mbox{Trace}_{M_{\p_0}/\Q_p}(a_j)\log_p(\sigma_0 hg_j\e)=0
\end{equation}
for all $h\in H$. Note $\mbox{Trace}(a_j)\in \Q_p$. We may assume at least one $a_j=1$, so $\mbox{Trace}(a_j)\neq 0$ for some $a_j$. After clearing denominators, Equation (\ref{eq3}) translates to: 
\[\sum_{j=1}^t b_j\log_p(\sigma_0 h g_j\e)=0\]
for all $h\in H$,
where $b_j\in\Z_p$ and at least one is non-zero.

There exists $b\in \Z$ 
 such that for all $j$ and for all $h$, the $bb_j\log_p(\sigma_0 hg_j\e)$ are in the region of $\C_p$ for which $\exp_p$ is defined. Hence \[\prod_{j=1}^t (\sigma_0 hg_j\e)^{bb_j}=1\]
for all $h\in H$.
So in  $M_{\p_0}$, 
$(hg_1\e)^{bb_{1}}\cdots (hg_t\e)^{bb_{t}}=1$
 for all $h\in H$. Lemma \ref{conv} implies that $(g_1\e)^{bb_{1}}\cdots (g_t\e)^{bb_{t}}=1$ in $M_\p$ for all $\p\in\Gamma$. Which contradicts the fact that $g_1\e,\ldots,g_t\e$ are $\Z_p$-independent.
\end{proof}

In summary the question posed in the introduction can be restated:

\medskip

\noindent {\bf Question.} What is the rank of the matrix $\left(\log_p(\sigma_i g_j\e)\right)_{\sigma_i\in S_p,\,g_j\in G}$?

\medskip

Recall that the conclusion of Leopoldt's Conjecture can be stated:
\[\rk\left(\log_{{p}}(\sigma_i g_j\e)\right)_{{\sigma_i\in E_p},\,g_j\in G}
=\rk\left(\log|\tau_i g_j\e|\right)_{{\tau_i\in E},\,g_j\in G}.\]
It seems reasonable that the answer to our question will involve removing rows from the matrix on the right-hand side. 

The remaining rows of the left-hand matrix are indexed by $S_p$, we need to define a subset of $E$ that will index the rows of the right-hand matrix.
 For each isomorphism  $\psi:\C_p\rightarrow\C$ consider the bijection between $E$ and $E_p$ as defined in  Lemma \ref{bij} and define  $S_\psi:=\{\psi.\sigma\,|\,\sigma\in S_p\}\subset E$.

We propose the following conjecture.

\begin{conj}\label{GLC}
Let $M$ be a finite Galois extension of $\Q$. Let $\e$ be a weak Minkowski unit $\e$ such that   $|\sigma_i g_j \e-1|_p<1$ for all $g_j\in G$ and $\sigma_i\in E_p$. Let $S_p$ be a subset of $E_p$. Let $\psi:\C_p\rightarrow\C$ be an isomorphism and for each $\psi$ let  $S_\psi$ be a subset of $E$ as defined previously. Then  
\[\rk(\log_p(\sigma_ig_j\e))_{\sigma_i\in S_p,\,g_j\in G}=\max_\psi\{\rk(\log|\tau_ig_j\e|)_{\tau_i\in S_\psi,\,g_j\in G}\}. \]
\end{conj}

When $S_p=E_p$ this conjecture is equivalent to Leopoldt's.

This paper's main result conditionally proves half of the conjecture, i.e., assuming the $p$-adic Schanuel Conjecture we prove that the left-hand side is greater than or equal to the right-hand side. Moreover we conditionally prove equality when $M$ is either totally real or CM over $\Q$.

\section{Schanuel's Conjecture}

S.~Schanuel formulated his conjecture in the 1960s. It is a generalization of several theorems that prove large classes of numbers are transcendental.
The { Hermite-Lindemann-Weirstrass Theorem} (1880s) implies that $e^\alpha$ is transcendental for $\alpha$ algebraic. The { Gel'fond-Schneider Theorem} (1934) implies that $\alpha^\beta$ is transcendental for $\alpha,\beta$ algebraic and $\beta\notin\Q$.  { Baker's Theorem} (1966) implies that $\alpha_1^{\beta_1}\cdots \alpha_n^{\beta_n}$ is transcendental for $\alpha_i,\beta_i$ algebraic, $\beta_i\notin\Q$, and $\beta_1,\ldots,\beta_n$ linearly independent over $\Q$.  Schanuel's Conjecture would imply the transcendence of many more numbers including $e+\pi,e\pi,\pi^e,e^e,\pi^\pi,\log\pi,$ and $(\log2)(\log3)$.

More generally, Schanuel's Conjecture is a result about the algebraic independence of sets of logarithms. It can be formulated for either complex or $p$-adic logarithms. 
\medskip

\noindent{\bf Schanuel's Conjecture: Logarithmic Formulation.}
{\em Let non-zero $\alpha_1, \ldots, \alpha_n$ be algebraic over $\Q$ and suppose that for any choice of the multi-valued logarithm $\log\alpha_1, \ldots, \log\alpha_n$ are linearly independent over $\Q$.  Then $\log\alpha_1, \ldots,$ $\log\alpha_n$ are algebraically independent over $\Q$.}

\medskip

\noindent {\bf $p$-adic  Schanuel Conjecture: Logarithmic Formulation.}
{\em Let non-zero $\alpha_1, \ldots, \alpha_n$ be algebraic over $\Q$ and  contained in an extension of $\Q_p$.  If $\log_p\alpha_1, \ldots, \log_p\alpha_n$ are linearly independent over $\Q$, then $\log_p\alpha_1, \ldots, \log_p\alpha_n$ are algebraically independent over $\Q$.}

\section{The Main Theorem}

The presentation in this section uses techniques proposed by M.~Waldschmidt \cite{Wald}. For the remainder of this paper we assume that
\begin{itemize}
\item {$M$ is a finite Galois extension of $\Q$ with $G=\Gal(M/\Q)$ and $|G|=n$.}
\item $\e\in \O_M^*$ is a weak Minkowski unit, i.e.,  $\{g_1\e,\ldots,g_n\e\}$ generates a finite index subgroup of $\O_M^*$. Moreover $|\sigma_i g_j \e-1|_p<1$ for all $\sigma_i\in E_p$ and $g_j\in G$.
\item $p$ is any fixed prime in $\Z$.
\item $S_p$ is a fixed subset of $E_p$.
\item $\psi$  denotes an isomorphism  from $\C_p$ to $\C$.
\item  $S_\psi:=\{\psi.\sigma\,|\,\sigma\in S_p\}\subset E$.

 \end{itemize}

\begin{thm}\label{rem} 
Let  $M$ be a finite Galois extension of $\Q$, with  $g_j\in G=\Gal(M/\Q)$, $\sigma_i\in E_p:=\Emb(M,\C_p)$, and $\tau_i\in E:=\Emb(M,\C)$.  Let
 $\e\in \O_M^*$ be a weak Minkowski unit such that $|\sigma_i g_j \e-1|_p<1$ for all $i, j$.
 Let $S_p$  and  $S_\psi$ be  as defined above. Assume the $p$-adic Schanuel Conjecture. Then 
\[\rk(\log_p(\sigma_ig_j\e))_{\sigma_i\in S_p,\,g_j\in G}\geq\max_\psi\{\rk(\log|\tau_ig_j\e|)_{\tau_i\in S_\psi,\,g_j\in G}\}.\]
\end{thm}

\begin{cor}
Let $\Gamma$, $\Delta_\Gamma$, $Y$, and $S_p$ be as defined in Section \ref{var}. Keep the notations and assumptions of the theorem. Then
\[\rk_{\Z_p}(\overline{\Delta_\Gamma(Y)})\geq\max_\psi\{\rk(\log|\tau_ig_j\e|)_{\tau_i\in S_\psi,\,g_j\in G}\}.\]
\end{cor}

We will need the following algebraic lemma. For its proof see \cite{Wald}.

\begin{lem}\label{poly}
Let $K$ be a field, $A_1,\ldots,A_t$ be elements in $K[T_1,\ldots,T_m]$, and $P\in K[T_1,\ldots,T_m,T_{m+1},\ldots,T_{m+t}]$. If the polynomial $P(T_1,\ldots,T_m,A_1,\ldots,A_t)$ is the zero polynomial in $K[T_1,\ldots,T_m]$, then $P$ is an element of the ideal of $K[T_1,\ldots,T_m,T_{m+1},\ldots,T_{m+t}]$ generated by the polynomials $\{T_{m+l}-A_l\}_{1\leq l\leq t}$.
\end{lem}

\begin{proof}[Proof of Theorem \ref{rem}.]
Set $\mathfrak{R}=(\log_p(\sigma_ig_j\e))_{\sigma_i\in S_p,\,g_j\in G}$ and $\mathfrak{r}=\rk(\mathfrak{R})$. Also for a fixed $\psi$ set $\mathfrak{l}=\rk(\log|\tau_ig_j\e|)_{\tau_i\in S_\psi,\,g_j\in G}$. For contradiction, assume that $\mathfrak{r}<\mathfrak{l}$.

Express the determinants of all $\mathfrak{l}\times \mathfrak{l}$ minors of a $|G|\times |S_p|$ matrix as polynomials in $|G|\,|S_p|=:v$ indeterminates: $P_k(T)\in\Q[T]$ (where $T=T_1, \ldots, T_{v}$). Since $\mathfrak{r}<\mathfrak{l}$,  when we evaluate the determinants for $\mathfrak{R}$ we have  $P_k(\log_p(\sigma_ig_j\e))=0$.

Let $F$ be the $\Q$-subvector space of $\C_p$ generated by the  entries of $\mathfrak{R}$ and define $m=\dim_\Q(F)$. Number the $v$ entries of the matrix by mapping $l\in\{1,\ldots,v\}$ to the entry $\log_p(\sigma_{i(l)}g_{j(l)}\e)$ so that the first $m$ entries $\left\{\log_p(\sigma_{i(s)}g_{j(s)}\e)\right\}_{1\leq s\leq m}$ are a basis of $F$. By hypothesis the $\{\sigma_{i(s)}g_{j(s)}\e\}_{1\leq s\leq m}$ are algebraic over $\Q$ and are in an extension of $\Q_p$. By choice the \linebreak $\{\log_p(\sigma_{i(s)}g_{j(s)}\e)\}_{1\leq s\leq m}$ are linearly independent over $\Q$. Thus the $p$-adic Schanuel  Conjecture implies that the $\{\log_p(\sigma_{i(s)}g_{j(s)}\e)\}_{1\leq s\leq m}$ are algebraically independent over $\Q$.

Let $a_{l},a_{ls}\in\Z$ with $a_{l}>0$ be such that for all $l$ (with  $1\leq l\leq v$)
\begin{equation}\label{sumlog2}
a_{l}\log_p(\sigma_{i(l)}g_{j(l)}\e)=\sum_{s=1}^m a_{ls}\log_p(\sigma_{i(s)}g_{j(s)}\e).
\end{equation}
Number the indeterminates so that $T_1,\ldots,T_m$ correspond to the basis elements of $F$.  
Using Equation (\ref{sumlog2}), define $A_l\in\Q[T_1,\ldots,T_m]$ as 
\[A_l=\sum_{s=1}^m \frac{a_{ls}}{a_l}T_s.\]
Now consider $P_k'(T_1,\ldots,T_m)=P_k(T_1,\ldots,T_m,A_{m+1},\ldots,A_{v})\in\Q[T_1,\ldots,T_m]$. We still have $P_k'\left(\log_p(\sigma_{i(s)}g_{j(s)}\e)\right)=0.$ But since the $\{\log_p(\sigma_{i(s)}g_{j(s)}\e)\}_{1\leq s\leq m}$ are algebraically independent over $\Q$,  $P_k'$ must be identically $0$ for all $k$. Thus Lemma \ref{poly} implies that, for all $k$,  $P_k$ is in the ideal generated by $\{T_l-A_l\}_{1\leq l\leq v}$, i.e., in the ideal generated by $\{T_l-\sum_{s=1}^m \frac{a_{ls}}{a_l}T_s\}_{1\leq l\leq v}$.

Let $a_{l},a_{ls}$ be as in Equation (\ref{sumlog2}). For each $l$ there exists $b\in\Z_{>0}$ such that for all $s$,
$ ba_{ls}\log_p(\sigma_{i(s)}g_{j(s)}\e)$ and $ba_l\log_p(\sigma_{i(l)}g_{j(l)}\e)$ are in the region of $\C_p$ for which $\exp_p$ is the inverse of $\log_p$.
 Thus, after multiplying both sides of Equation (\ref{sumlog2}) by $b$, we can apply the $p$-adic exponential function and
\[(\sigma_{i(l)}g_{j(l)}\e)^{a_l b}=\prod_{s=1}^m (\sigma_{i(s)}g_{j(s)}\e)^{a_{ls}b}.\]
For an isomorphism $\psi:\C_p\rightarrow\C$, Lemma \ref{bij} gives us a bijection between the elements of $E_p$ and $E$: 
\[\xymatrix{\tau_i:M\ar^{\sigma_i}[r]&\C_p\ar^{\psi}_{\cong}[r]&\C}. \]
 So
\begin{align*}
\psi(\sigma_{i(l)}g_{j(l)}\e)^{a_l b}&=\prod_{s=1}^m \psi(\sigma_{i(s)}g_{j(s)}\e)^{a_{ls}b}\\
(\tau_{i(l)}g_{j(l)}\e)^{a_l b}&=\prod_{s=1}^m (\tau_{i(s)}g_{j(s)}\e)^{a_{ls}b}.
\end{align*}
Next apply the usual complex modulus to both sides:
\[|\tau_{i(l)}g_{j(l)}\e|^{a_l b}=\prod_{s=1}^m |\tau_{i(s)}g_{j(s)}\e|^{a_{ls}b}.\]
Then apply the real logarithm  to both sides:
\begin{equation}\label{cx2}
a_l\log\left|\tau_{i(l)}g_{j(l)}\e\right|=\sum_{s=1}^m a_{ls}\log\left|\tau_{i(s)}g_{j(s)}\e\right|.\end{equation}

   Equation (\ref{cx2}) combined with the fact that all $P_k$ are in the ideal  generated by $\{T_l-\sum_{s=1}^m \frac{a_{ls}}{a_l}T_s\}_{1\leq l\leq v}$ implies that for all $k$, $$P_k(\log\left|\tau_{i(1)}g_{j(1)}\e\right|,\ldots,\log\left|\tau_{i(v)}g_{j(v)}\e\right|)=0.$$ But $\rk \left(\log|\tau_ig_j\e|\right)_{\tau_i\in S_\psi,\,g_j\in G}=\mathfrak{l}$. Thus the determinant of at least one $\mathfrak{l}\times \mathfrak{l}$ minor is non-zero. Hence we have a contradiction.
   
   So for every $\psi$ we have $\rk\mathfrak{R}\geq \rk(\log|\tau_ig_j\e|)_{\tau_i\in S_\psi,\,g_j\in G}$.    Despite the existence of infinitely many isomorphisms between $\C_p$ and $\C$, there are only finitely many subsets $S_\psi$ of the finite set $E$. Thus $\rk\mathfrak{R}\geq\max_\psi\{\rk(\log|\tau_ig_j\e|)_{\tau_i\in S_\psi,\,g_j\in G}\}.$
   \end{proof}

\begin{thm}\label{SchL} For a finite Galois extension of $\Q$, the $p$-adic Schanuel Conjecture implies Leopoldt's Conjecture.
\end{thm}
   
   \begin{proof}
   For Leopoldt's Conjecture, $S_p=E_p$ and so for all $\psi$, $S_\psi=E$. So Theorem \ref{rem} implies 
   \[\rk(\log_p(\sigma_ig_j\e))_{\sigma_i\in E_p,\,g_j\in G}\geq\rk(\log|\tau_ig_j\e|)_{\tau_i\in E,\,g_j\in G}.\]

   We prove the reverse inequality directly.    
By Corollary \ref{r}, $\rk_\Z(\O_M^*)= \rk(\log|\tau_ig_j\e|)_{\tau_i\in E,\,g_j\in G}$.   Let $X$ and $\Delta$ be as defined in Section \ref{2ways}. $X$ is a finite index subgroup of $\O_M^*$ and so $\rk_\Z(X) = \rk_\Z(\O_M^*)$. Since $\Z\subset\Z_p$,  $\rk_\Z(X)\geq\rk_{\Z_p}(\overline{\Delta X})$. Proposition \ref{equivV} implies \[\rk_{\Z_p}(\overline{\Delta X})=\rk(\log_p(\sigma_ig_j\e))_{\sigma_i\in E_p,\,g_j\in G}.\] Together these inequalities and equalities imply   
 \[\rk(\log|\tau_ig_j\e|)_{\tau_i\in E,\,g_j\in G}\geq\rk(\log_p(\sigma_ig_j\e))_{\sigma_i\in E_p,\,g_j\in G}.\]
   \end{proof}

   \section{Computing Rank}\label{calc}

In the following subsections we will compute a lower bound for \linebreak $\rk(\log_p(\sigma_ig_j\e))_{\sigma_i\in S_p,\,g_j\in G}$ and in some cases we will compute  it exactly. We do this by looking more closely at $S_p$, or more precisely at the removed rows $E_p\smallsetminus S_p$. 
   
Fix $p$ and $S_p$. Let $$\mathfrak{R}=(\log_p(\sigma_ig_j\e))_{\sigma_i\in S_p,\,g_j\in G}.$$ Let $r$ be the $\Z$-rank of the global units, i.e.,  $r= \rk(\log|\tau_ig_j\e|)_{\tau_i\in E,\,g_j\in G}$.  For $x\in\R$, define 
  \[x^+=\begin{cases}
    x  & x\geq0, \\
   0   & x<0.
\end{cases}\]
   
   \subsection{Real Extensions}

  \begin{prop}\label{real}
Let $M$ be a totally real Galois extension of $\Q$. Let $t=|E_p \smallsetminus S_p|$. Assume the $p$-adic Schanuel Conjecture. Then \[\rk(\log_p(\sigma_ig_j\e))_{\sigma_i\in S_p,\,g_j\in G}=r-(t-1)^+.\]
  \end{prop}
  
  \begin{proof}
  If $t=0$, then we are in the case of Leopoldt's Conjecture and $\rk\mathfrak{R}=r$ and the formula holds. So we now assume $t>0$.
  
  Clearly, for all $\psi$, $|E \smallsetminus S_\psi|=|E_p \smallsetminus S_p|=t$.
  
 We showed in the proof of Proposition \ref{unit} and in Corollary \ref{r} that the $r\times r$ matrix $(\log|\tau_ig_j\e|)_{i,j=1,\ldots,r}$  and the $n\times n$ matrix $(\log|\tau_ig_j\e|)_{\tau_i\in E,\,g_j\in G}$ both have  rank equalling $r$.
Recall that $r=n-1$ for totally real Galois extensions. Thus for all $\psi$ 
 \[\rk(\log|\tau_ig_j\e|)_{\tau_i\in S_\psi,\,g_j\in G}=(r+1)-t=r-(t-1).\]
Hence Theorem \ref{rem} implies $\rk\mathfrak{R}\geq r-(t-1)$.
But since $\mathfrak{R}$ has $n-t=(r+1)-t=r-(t-1)$ rows, $\rk\mathfrak{R}\leq r-(t-1)$. The proposition follows.
  \end{proof}
  
  \begin{cor}
Let $\Gamma$, $\Delta_\Gamma$, and $Y$ be as defined in Section \ref{var}. Keep the notations and assumptions of the proposition. Then
\[\rk_{\Z_p}(\overline{\Delta_\Gamma(Y)})=r-(t-1)^+.\]
\end{cor}

   \subsection{Complex Extensions}
Let $M$ be a {complex} Galois extension of $\Q$. Recall that we have fixed $S_p\subset E_p$ and for each isomorphism $\psi:\C_p\rightarrow\C$ there is a corresponding $S_\psi\subset E$. Throughout $\e$ is a weak Minkowski unit such that $|\sigma_i g_j \e-1|_p<1$ for all $g_j\in G$ and $\sigma_i\in E_p$.

    \begin{prop}\label{L}
The rank of \[\mathfrak{L}_\psi:=(\log|\tau_ig_j\e|)_{\tau_i\in S_\psi,\,g_j\in G}\]
is determined by the number of pairs of complex conjugate embeddings in $E\smallsetminus S_\psi$. In particular, if $v_\psi$ is the number of pairs of complex conjugate embeddings in $E\smallsetminus S_\psi$ then $\rk\mathfrak{L}_\psi=r-(v_\psi-1)^+$.
 \end{prop}
   
\begin{proof} 
If there are $v_\psi$ pairs of complex conjugate embeddings in $E\smallsetminus S_\psi$ then there are exactly $\frac{n}{2}-v_\psi$ embeddings $\{\tau_1,\ldots,\tau_{n/2-v_\psi}\}\in S_\psi$ such that $\tau_j\neq \overline{\tau_i}$ for all $i,j\in\{1,\ldots,\frac{n}{2}-v_\psi\}.$ We proved in Section \ref{nota} that  \[\rk(\log|\tau_i g_j\e|)_{i=1,\ldots,r;\,g_j\in G}=r=\frac{n}{2}-1 \] when $E$ is ordered $E=\{\tau_1,\ldots,\tau_{r+1},\overline{\tau_1},\ldots,\overline{\tau_{r+1}}\}$.  Thus, if $v_\psi>0$, the rows of $\mathfrak{L}_\psi$ that  correspond to $\tau_1,\ldots,\tau_{n/2-v_\psi}$ are linearly independent. Any other row corresponding to some $\tau_i\in S_\psi\smallsetminus\{\tau_1,\ldots,\tau_{n/2-v_\psi}\}$ is dependent via the relation $\log|\tau_i g_j\e|=\log|\overline{\tau_i}g_j\e|$. Thus  
\[\rk\mathfrak{L}_\psi=\begin{cases}
  \frac{n}{2}-1    & \text{if } v_\psi=0, \\
  \frac{n}{2}-v_\psi    & \text{if } v_\psi>0.
\end{cases} \]
So $\rk\mathfrak{L}_\psi=r-(v_\psi-1)^+$.
\end{proof}

   Fix $\sigma \in E_p$. Then we can write $E_p\smallsetminus S_p=\{\sigma g_1,\ldots,\sigma g_k\}$ for  some $g_i\in G$. So for each isomorphism $\psi$,  
   \begin{align*}
   E\smallsetminus S_\psi & = \{\psi.\sigma g_1,\ldots,\psi.\sigma g_k\}\\
   &= \{\tau_\psi g_1,\ldots,\tau_\psi g_k\},\; \tau_\psi\in E.
   \end{align*}
   On the other hand, for each $\tau \in E$ there exists an isomorphism $\psi_\tau:\C_p\rightarrow\C$ such that $\psi_\tau.\sigma=\tau$. (The proof of the existence of $\psi_\tau$ is similar to the proof of the existence of an isomorphism between $\C_p$ and $\C$.) Thus for each $\tau\in E$ there exists an isomorphism $\psi_\tau$ such that $\{\tau g_1,\ldots, \tau g_k\}=E\smallsetminus S_{\psi_\tau}$. Hence the set of sets $\{\tau g_1,\ldots, \tau g_k\}_{\tau\in E}$ equals the set of sets $\{E\smallsetminus S_\psi\}_\psi$. This is  independent of the choice of $\sigma$.
   
   Therefore, despite there being infinitely many isomorphisms $\psi$, there are exactly $n$ different $E\smallsetminus S_\psi$. Hence there are only finitely many $v_\psi$, where $v_\psi$ is the number of pairs of complex conjugate embeddings in $E\smallsetminus S_\psi$. So we can define $v=\min_\psi\{v_\psi\}$.

  \begin{cor} Let $M$ be a complex Galois extension of $\Q$.
Let $v=\min_\psi\{v_\psi\}$. Assume the $p$-adic Schanuel Conjecture. 
  Then, 
\[\rk\left(\log_p(\sigma_ig_j\e)_{\sigma_i\in S_p,\,g_j\in G}\right)
\geq r-(v-1)^+.\]
  \end{cor}

 \begin{proof}
This is immediate from Proposition \ref{L} and Theorem \ref{rem}.
\end{proof}

  This result depends on examining each $E\smallsetminus S_\psi$ to calculate $v_\psi$. We can get a better result that only requires us to consider $E_p\smallsetminus S_p$. First we need a lemma relating pairs of complex conjugate embeddings to elements of $G$ induced by complex conjugation.
   
       \begin{lem}\label{cc=c} Let $\tau\in E$ and $c\in G$ be such that $\overline{\tau}=\tau c$. Then the number of pairs of complex conjugate embeddings in $\{\tau g_1,\ldots, \tau g_k\}$ equals the number of right cosets of $\{\id,c\}$ in $\{g_1,\ldots,g_k\}$.
 \end{lem}
   
   \begin{proof}
   Let $\{g,cg\}$  be a right coset in $\{g_1,\ldots,g_k\}$. Then $\{\tau g,\tau cg\}$ are embeddings in $\{\tau g_1,\ldots, \tau g_k\}$. For all $m\in M$, 
   \begin{align*}
   \tau cg(m)&=\tau c(g(m))\\
   &=\overline{\tau(g(m))}\\
   &=\overline{\tau g(m)}
   \end{align*}
Thus $\tau g$ and $\tau cg$ are a pair of complex conjugate embeddings.
   
   Let $\tau g_j$ and $\tau g_i$ be a pair of complex conjugate embeddings in $\{\tau g_1,\ldots, \tau g_k\}$, with $g_j,g_i\in \{g_1,\ldots,g_k\}$. Then for all $m\in M$ 
   \begin{align*}
   \tau g_j(m)
   &=\overline{\tau g_i(m)}\\
   &=\overline{\tau(g_i(m))}\\
   &=\tau c(g_i(m))\\
   &=\tau c g_i(m).
   \end{align*}
  Since $\tau$ is injective,  $g_j(m)=cg_i(m)$ for all $m\in M$. Hence $g_j=cg_i$. Thus $\{g_j,g_i\}$ is a right coset of $\{\id,c\}$.
  
  Since every coset determines a unique pair of complex conjugates and vice versa, the lemma is proven.
   \end{proof}

   Pick $\sigma \in E_p$ and write $E_p\smallsetminus S_p=\{\sigma g_1,\ldots,\sigma g_k\}$ for  some $g_i\in G$. Define $J_\sigma:=\{g_1,\ldots,g_k\}$. Let $C$ be the set of  elements in $G$ induced by complex conjugation. Then for $c\in C$, define $t_c$ to be the number of right cosets of $\{\id,c\}$ in $J_\sigma$ and define  $t=\min_{c\in C}\{t_c\}$.
   
   \begin{prop}\label{horse}  Let $M$ be a complex Galois extension of $\Q$.   
 Define $v_\psi$ to be the number of pairs of complex conjugate embeddings in $E\smallsetminus S_\psi$, $v:=\min_\psi\{v_\psi\}$, $t_c$ to be the number of right cosets of $\{\id,c\}$ in $J_\sigma$, and  $t:=\min_{c\in C}\{t_c\}$.
   Then $t$ is independent of the choice of $\sigma$ and 
   $v=t.$
   \end{prop}
   
   \begin{proof} Choose $\sigma \in E_p$ and write $E_p\smallsetminus S_p=\{\sigma g_1,\ldots,\sigma g_k\}$. Then $J_\sigma=\{g_1,\ldots,g_k\}$.
   
   For each isomorphism $\psi$ there exists $\tau_\psi\in E$ such that $E\smallsetminus S_\psi=\{\tau_\psi g_1,\ldots, \tau_\psi g_k\}$. Let $c_\psi\in C$ be such that $\overline{\tau_\psi}=\tau_\psi c_\psi$. Then Lemma \ref{cc=c} implies that $v_\psi=t_{c_\psi}$.
   
   On the other hand, for each $c\in C$ there exists $\tau\in E$ such that $\bar{\tau}=\tau c$, and for this $\tau$ there exists an isomorphism $\psi$ such that $E\smallsetminus S_\psi=\{\tau g_1,\ldots,\tau g_k\}$. Again Lemma \ref{cc=c} implies that $t_c=v_\psi$.
   
   Hence $\min_\psi\{v_\psi\}=\min_{c\in C}\{t_c\}$. So $v=t$. Since $v$ is independent of $\sigma$, so is $t$.
      \end{proof}

  \begin{cor}\label{tt} Let $M$ be a complex Galois extension of $\Q$. Let $t=\min_{c\in C}\{t_c\}$ (as defined above). Assume the $p$-adic Schanuel Conjecture.
  Then 
\[\rk\left(\log_p(\sigma_ig_j\e)_{\sigma_i\in S_p,\,g_j\in G}\right)
\geq r-(t-1)^+.\]
  \end{cor}

   \begin{cor}\label{ss}
 Let $M$ be a complex Galois extension of $\Q$. Let $\Gamma$, $\Delta_\Gamma$, and $Y$ be as defined in Section \ref{var}. 
Let $t=\min_{c\in C}\{t_c\}$. Assume the $p$-adic Schanuel Conjecture. Then
\[\rk_{\Z_p}(\overline{\Delta_\Gamma(Y)})\geq r-(t-1)^+.\]
\end{cor}

   \subsection{The CM Case}

We continue to  let $\e$ be a weak Minkowski unit such that $|\sigma_i g_j \e-1|_p<1$ for all $g_j\in G$ and $\sigma_i\in E_p$.
 
 {If   we assume that $M$ is CM over $\Q$,} then Theorem \ref{rem} and Corollaries \ref{tt} and \ref{ss} can be strengthened in two ways: 
we can remove the ``maximum'' and ``minimum'' qualifications and we can replace the inequalities with equalities. 

   The keys to removing the ``maximum'' and ``minimum'' qualifications are the following  Lemma and the next Proposition. 

       \begin{lem}\label{uniq}
Let $M$ be CM over $\Q$. For all $\tau\in E$ there exists a {\em unique} $c\in G$ such that the following diagram commutes. 
\[\xymatrix{M\ar_c[d]\ar^{\overline{\tau}}[dr]\\
M\ar_{\tau}[r]&\C
}\]
\end{lem}

\begin{prop}\label{ind} Let $M$ be a CM extension of $\Q$. 
Then $\rk(\log|\tau_ig_j\e|)_{\tau_i\in S_\psi,g_j\in G}$ is independent of the isomorphism $\psi:\C_p\rightarrow \C$.
\end{prop}  

 \begin{proof} Write $E_p\smallsetminus S_p=\{\sigma g_1,\ldots,\sigma g_k\}$ and $J_\sigma:=\{g_1,\ldots,g_k\}$. In the proof of Proposition \ref{horse}, we proved that for all $v_\psi$ there exists $c\in C$ such that $v_\psi=t_c$. But since $M$ is CM over $\Q$, there is only one element in $C$, denote it $c$. So for all isomorphisms $\psi$ 
 \[\rk(\log|\tau_ig_j\e|)_{\tau_i\in S_\psi,g_j\in G}=r-(v_\psi-1)^+=r-(t-1)^+,\]
 where $t=t_c$ is the number of right cosets of $\{\id,c\}$ in $J_\sigma$.
\end{proof}

The key to replacing the inequalities with equalities is the following lemma.
 
 \begin{lem}\label{rtun} Let $M$ be a CM extension of $\Q$ with $m_1,m_2\in \mathcal{O}_M^*$ and let $\tau_i,\tau_j\in E$.
If 
\[|  \tau_im_1   | = |\tau_jm_2    |\]
where $|\cdot|$ is the usual complex modulus, then 
\[\zeta( \tau_im_1  )=   \tau_jm_2  \]
where $\zeta$ is a root of unity.
 \end{lem}
 
 \begin{proof}
  
Since $M$ is a Galois extension, $\tau_i(M)=\tau_j(M)$. Denote this subfield of $\C$ as $K$. Since $M$ is isomorphic to $K$, 
 $K$ is a CM extension of $\Q$ having the same degree over $\Q$ as $M$ does. Let $n$ be  the degree of $K$ over $\Q$. Let $K^+$ be the totally real subfield of $K$. Hence $K$ is a degree 2 extension of $K^+$ and $K^+$ is a degree $\frac{1}{2}n$  extension of $\Q$.

 Dirichlet's Unit Theorem implies $\rk_\Z\mathcal{O}_K^*=\frac{1}{2}n-1$ because $K$ is a complex extension of $\Q$ and $\rk_\Z \mathcal{O}_{K^+}^{*}=\frac{1}{2}n-1$ because $K^+$ is a totally real extension of $\Q$. Since $\mathcal{O}_{K}^{*}$ and $\mathcal{O}_{K^+}^{*}$ have the same finite rank and $\mathcal{O}_{K^+}^{*}$ is a subgroup of $\mathcal{O}_{K}^{*}$, $\mathcal{O}_{K}^{*}\big/\mathcal{O}_{K^+}^{*}$ is a finite group.

The following series of calculations on elements of $K$ shows that given $|  \tau_im_1   | = |\tau_jm_2    |$ there exists $\zeta$ with $|\zeta|=1$ such that $\zeta( \tau_im_1  )=   \tau_jm_2$. ($c$ is as defined in Lemma \ref{uniq}.)
 \begin{align*}
 |  \tau_im_1   | &= |\tau_jm_2    |\\
 (\tau_icm_1)( \tau_im_1 )   &=(\tau_jcm_2)( \tau_jm_2)\\
   (\tau_jcm_2)^{-1}(\tau_icm_1)( \tau_im_1 )   &=( \tau_jm_2)
 \end{align*}
 Let $\zeta = (\tau_jcm_2)^{-1}(\tau_icm_1)$.
  \begin{align*}   
 \zeta( \tau_im_1  )&=   \tau_jm_2\\
| \zeta|\,| \tau_im_1  |&=  | \tau_jm_2 |\\
|\zeta|&=1
 \end{align*}

 Since $m_1,m_2\in \mathcal{O}_M^*$, $\tau_im_1,   \tau_jm_2 \in \mathcal{O}_{K}^{*}$. Thus $\zeta\in\mathcal{O}_{K}^{*}$. Since $\mathcal{O}_{K}^{*}\big/\mathcal{O}_{K^+}^{*}$ is a finite group, there exists $q\in\Z$ such that $\zeta^q\in \mathcal{O}_{K^+}^{*}\subset\R$. Thus since $|\zeta|=1$, $|\zeta^q|=1$. Combining this with $\zeta^q\in \R$ we have $\zeta^q=\pm1$.
  \end{proof}

\begin{thm}\label{cm} Let $M$ be CM over $\Q$. Let $\e\in \O_M^*$ be a weak Minkowski unit such that $|\sigma_i g_j \e-1|_p<1$ for all $i,j$. Let $S_p$ be a subset of $E_p$ and $S_\psi$ be a subset of $E$ as defined previously. Assume both Schanuel's Conjecture and the $p$-adic version. Then for any isomorphism $\psi:\C_p\rightarrow\C$ 
\[\rk(\log_p(\sigma_ig_j\e))_{\sigma_i\in S_p,\,g_j\in G}=\rk(\log|\tau_ig_j\e|)_{\tau_i\in S_\psi,\,g_j\in G}. \]
\end{thm}   

Before we begin the proof, note the following important corollaries.

\begin{cor}\label{CMt} Keep the assumptions of the previous theorem. For any choice of $\sigma$, define $J_\sigma$  as before. Let $c$ be the unique element of $G$ induced by complex conjugation and let $t$ be the number of right cosets of $\{\id, c\}$ in $J_\sigma$. Then
\[\rk(\log_p(\sigma_ig_j\e))_{\sigma_i\in S_p,g_j\in G}=r-(t-1)^+.\]
\end{cor}

   \begin{cor}
Let $\Gamma$, $\Delta_\Gamma$, and $Y$ be as defined in Section \ref{var}. Keep the  assumptions of the previous Theorem and Corollary. Then
\[\rk_{\Z_p}(\overline{\Delta_\Gamma(Y)})= r-(t-1)^+.\]
\end{cor}

\begin{proof}[Proof of Theorem \ref{cm}] Let $\mathfrak{R}=(\log_p(\sigma_ig_j\e))_{\sigma_i\in S_p,\,g_j\in G}$ and $\mathfrak{L}_\psi=$\linebreak $(\log|\tau_ig_j\e|)_{\tau_i\in S_\psi,\,g_j\in G}.$
 Define  $\mathfrak{r}:=\rk\mathfrak{R}$. By Proposition \ref{ind} the rank of  $\mathfrak{L}_\psi$ is the same for all  $\psi$, we will denote this rank by $\mathfrak{l}$.  By Theorem \ref{rem} we know that   $\mathfrak{r}\geq\mathfrak{l}$. 
 For contradiction, assume that $\mathfrak{r}>\mathfrak{l}$. 

 This proof is very similar to the proof of Theorem \ref{rem}. 
 
Express the determinants of all $\mathfrak{r}\times \mathfrak{r}$ minors of a $|G|\times |S_p|$ matrix as polynomials in $|G|\,|S_p|=:v$ indeterminates, $P_k(T)\in\Q[T]$ (where $T=T_1, \ldots, T_{v}$). Since $\mathfrak{r}>\mathfrak{l}$,  when we evaluate the polynomials for $\mathfrak{L}_\psi$ we have  $P_k(\log|\tau_ig_j\e|)=0$.

Let $F$ be the $\Q$-subvector space of $\R$ generated by the  entries of $\mathfrak{L}_\psi$ and define $m=\dim_\Q(F)$. Number the $v$ entries of the matrix by mapping $l\in\{1,\ldots,v\}$ to the entry $\log|\tau_{i(l)}g_{j(l)}\e|$ so that the first $m$ entries $\left\{\log|\tau_{i(s)}g_{j(s)}\e|\right\}_{1\leq s\leq m}$ are a basis of $F$.  Since $M$ is a finite extension of $\Q$, $\e\in M$ is algebraic over $\Q$, hence so are $\tau_{i(s)}g_{j(s)}\e$  and $|\tau_{i(s)}g_{j(s)}\e|$. Thus the $\{|\tau_{i(s)}g_{j(s)}\e|\}_{1\leq s\leq m}$ satisfy the hypotheses of Schanuel's Conjecture (NOT the $p$-adic conjecture). So the $\{\log|\tau_{i(s)}g_{j(s)}\e|\}_{1\leq s\leq m}$ are algebraically independent over $\Q$.

Let $a_{l},a_{ls}\in\Z$ with $a_{l}>0$ be such that  for all $l$,  $1\leq l\leq v$,
\begin{equation}\label{sumlog3}
a_{l}\log|\tau_{i(l)}g_{j(l)}\e|=\sum_{s=1}^m a_{ls}\log|\tau_{i(s)}g_{j(s)}\e|.
\end{equation}
Number the indeterminates so that $T_1,\ldots,T_m$ correspond to the basis elements of $F$. 
Using Equation (\ref{sumlog3}), define $A_l\in\Q[T_1,\ldots,T_m]$ as 
\[A_l=\sum_{s=1}^m \frac{a_{ls}}{a_l}T_s.\]
Now consider $P_k'(T_1,\ldots,T_m)=P_k(T_1,\ldots,T_m,A_{m+1},\ldots,A_{v})\in\Q[T_1,\ldots,T_m]$. We still have $P_k'\left(\log|\tau_{i(s)}g_{j(s)}\e|\right)=0.$ Since the $\{\log|\tau_{i(s)}g_{j(s)}\e|\}_{1\leq s\leq m}$ are algebraically independent over $\Q$,  $P_k'$ must be identically $0$ for all $k$. Thus Lemma \ref{poly} implies that, for all $k$,  $P_k$ is in the ideal generated by \linebreak $\{T_l-A_l\}_{1\leq l\leq v}$, i.e., in the ideal generated by $\{T_l-\sum_{s=1}^m \frac{a_{ls}}{a_l}T_s\}_{1\leq l\leq v}$.

Equation (\ref{sumlog3}) is true in $\R$, so we apply the real exponential function 
\begin{equation*}\label{93}|\tau_{i(l)}g_{j(l)}\e|^{a_l }=\prod_{s=1}^m |\tau_{i(s)}g_{j(s)}\e|^{a_{ls}}.\end{equation*} 
Then Lemma \ref{rtun} implies that  
\[\zeta(\tau_{i(l)}g_{j(l)}\e)^{a_l }=\prod_{s=1}^m (\tau_{i(s)}g_{j(s)}\e)^{a_{ls}},\]
where $\zeta$ is a root of unity in $\C$.
For an isomorphism $\phi:\C\rightarrow\C_p$, Lemma \ref{bij} gives us a bijection between the elements of $E_p$ and $E$: 
\[\xymatrix{\sigma_i:M\ar^(.6){\tau_i}[r]&\C_p\ar^{\phi}_{\cong}[r]&\C}. \]
 So  
\begin{align*}
\phi(\zeta)\phi(\tau_{i(l)}g_{j(l)}\e)^{a_l }&=\prod_{s=1}^m \phi(\tau_{i(s)}g_{j(s)}\e)^{a_{ls}}\\
\phi(\zeta)(\sigma_{i(l)}g_{j(l)}\e)^{a_l }&=\prod_{s=1}^m (\sigma_{i(s)}g_{j(s)}\e)^{a_{ls}}.
\end{align*}
Next we apply the $p$-adic logarithm  to both sides, noting that $\phi(\zeta)$ is a root of unity in $\C_p$ and that the $p$-adic logarithm of a root of unity equals 0:
\begin{align}
\log_p(\phi(\zeta))+a_l\log_p\left(\sigma_{i(l)}g_{j(l)}\e\right)&=\notag\\a_l\log_p\left(\sigma_{i(l)}g_{j(l)}\e\right)&=\sum_{s=1}^m a_{ls}\log_p\left(\sigma_{i(s)}g_{j(s)}\e\right).\label{cx3}
\end{align}
   
   Now Equation (\ref{cx3}) combined with the fact that all $P_k$ are in the ideal  generated by $\{T_l-\sum_{s=1}^m \frac{a_{ls}}{a_l}T_s\}_{1\leq l\leq v}$ implies that for all $k$, $$P_k(\log_p\left(\sigma_{i(1)}g_{j(1)}\e\right),\ldots,\log_p\left(\sigma_{i(v)}g_{j(v)}\e\right))=0.$$ But $\rk \mathfrak{R}=\mathfrak{r}$. Thus the determinant of at least one $\mathfrak{r}\times \mathfrak{r}$ minor is non-zero. Hence we have a contradiction. 
\end{proof}

\end{document}